\documentclass[11pt,a4paper]{amsart}
\usepackage{amsfonts}
\usepackage{epsfig}
\usepackage{graphicx}
\usepackage{amsmath}
\usepackage{amssymb}
\usepackage{color}
\usepackage{mathrsfs}
\usepackage{soul}
\usepackage{txfonts}

\newtheorem{theorem}{Theorem}[section]
\newtheorem{proposition}[theorem]{Proposition}
\newtheorem{lemma}[theorem]{Lemma}

\newtheorem{remark}{Remark}[section]

\newtheorem{maintheorem}{Theorem}

\newtheorem{maincorollary}{Corollary}

\newcommand{\Rr}{\mathbb R}
\newcommand{\Zz}{\mathbb Z}

\newcommand{\R}{\mathbb{R}}

\newcommand{\Z}{\mathbb{Z}}
\newcommand{\Tt}{{\mathbb{T}}}
\newcommand{\Aa}{{\mathbb{A}}}

\newcommand{\LL}{\mathcal L}

\newcommand{\RR}{{\mathcal R}}

\newcommand{\tr}{{\rm tr}}

\DeclareMathAlphabet{\mathpzc}{OT1}{pzc}{m}{it}

\newcommand{\comment}[1]{}

\title[Shadowing and hyperbolicity for geodesic flows]
{
On shadowing and hyperbolicity for geodesic flows on surfaces}

\author[M. Bessa]{M\'{a}rio Bessa}

\address{Universidade da Beira Interior, 
Rua Marqu\^es d'\'Avila e Bolama,
6201-001 Covilh\~a, Portugal}
\email{bessa@ubi.pt}

\author[J. Lopes Dias]{Jo\~{a}o Lopes Dias}
\address{Departamento de Matem\'atica and CEMAPRE, ISEG, 
Universidade de Lisboa,
Rua do Quelhas 6, 
1200-781 Lisboa, Portugal}
\email{jldias@iseg.ulisboa.pt}

\author[M. J. Torres]{Maria Joana Torres}
\address{CMAT and Departamento de Matem\'atica e Aplica\c{c}\~{o}es, 
Universidade do Minho, 
Campus de Gualtar, 
4700-057 Braga, Portugal}
\email{jtorres@math.uminho.pt}

\date{\today}

\begin{document}


\begin{abstract}
We prove that the geodesic flow on closed surfaces displays a hyperbolic set if the shadowing property holds $C^2$-robustly on the metric. Similar results are obtained when considering even feeble properties like the weak shadowing and the specification properties. Despite the Hamiltonian nature of the geodesic flow, the arguments in the present paper differ completely from those used in ~\cite{BRT} for Hamiltonian systems.
\end{abstract}

\maketitle


\noindent
\textbf{\textup{2010} Mathematics Subject Classification:}
Primary: 53D25, 37D30, 37C50  Secondary: 37J10, 37C75.
37D30.


\section{Introduction}\label{introduction}

The geodesic flow associated to a given metric describes the flow trajectory of a  free particle not subject to external forces. 
When studying the geodesic flow associated to negative curvature manifolds, Anosov discovered a surprising and quite rigid geometrical and dynamical property: uniform hyperbolicity (\cite{A}).
Its main characteristic is the uniform rate of contraction and expansion of the invariant directions under the tangent flow.

Uniform hyperbolicity turned out to be a fundamental ingredient for the understanding of general dynamical systems (see e.g. \cite{S1}). 
It allowed the construction of a fruitful ge\-o\-me\-tric theory of invariant manifolds, a stability theory (uniform hyperbolicity is essentially equivalent to structural stability),
a statistical theory (smooth ergodic theory) and a numerical theory (shadowing and expansiveness).

It was however realized from an early stage that uniform hyperbolicity was far from covering the complete scope of possible dynamical behaviours, even in a typical sense.
A mechanical system which is uniformly hyperbolic on each positive energy level is due to Hunt and MacKay \cite{HM}, but other examples are scarse.
Other more flexible definitions of hyperbolicity began to arise like nonuniform hyperbolicity, partial hyperbolicity  (see e.g. ~\cite{CPuj} in the context of geodesic flows) and dominated splitting.

Uniform hyperbolicity was found to yield very interesting numerical properties unlike other kind of systems. For instance, the \emph{shadowing property} holds for uniformly hyperbolic systems, i.e. almost orbits affected with a certain error are aproximated by true orbits. This amazing property not present under partial hyperbolicity ~\cite{BDT}, contains in itself much of the rigidity of the strong assumptions of uniform hyperbolicity. 
Two other quite important properties are the specification (\S\ref{spec}) and the weak shadowing (\S\ref{wsp}) properties, which on the other hand can appear in partial hyperbolic systems.

The case of Hamiltonian systems yields some surprising consequences arising from numerical properties.
If we assume $C^2$-robustness of shadowing, 
then the closure of the periodic points is a uniformly hyperbolic set~\cite{BRT}. 
Thus, due to the general density theorem for Hamiltonians~\cite{PugRob},  the Hamiltonian is Anosov. 
In other words, stability of a numerical property allow us to obtain geometrical, dynamical and also topological knowledge.

A natural question is whether these results are extensible to the subclass of Hamiltonians formed by the geodesic flows.

In this paper we restrict our study to surfaces and show that the robustness of the shadowing also implies that the closure of the periodic points is a uniformly hyperbolic set (Theorem~\ref{thm1}).

The perturbation tools for geodesic flows are very delicate as opposed to the general Hamiltonian case. 
We can only perturb the metric, hence the perturbation is never a local issue in phase space. 
Furthermore, the hyperbolic structure of the closure of the periodic orbits cannot be extrapolated to the whole energy level due to the absence of a closing lemma for geodesic flows. 
As a consequence we are not able to assure global hyperbolicity.
Other techniques not available in the geodesic flow context are the so-called pasting lemma and suspension theorem which were crucial in ~\cite{BRT}.

Our proof of ~Theorem~\ref{thm1} relies on properties of the geodesic flow on surfaces, in particular on the twist property of the Poincar\'e section map about a closed geodesic. 
The existence of invariant curves surrounding elliptic points of area-preserving twist maps implies the spliting of the two-dimensional section into invariant disjoint open sets, thus forbiding shadowing.
So, under stability of shadowing the periodic points can only be hyperbolic. 
Finally, a well-known result by Contreras and Paternain~\cite{CP} guarantees that the closure of the periodic points is a uniformly hyperbolic set.

In \S\ref{section geodesic flow} we introduce the Hamiltonian geodesic flow and the Poincar\'e section map.
The definitions of the shadowing properties appear in \S\ref{section: shadowing}.
The main results are stated in \S\ref{State} and proved in \S\ref{proofs}.
In \S\ref{section applications} we present some interesting applications.



\section{Geodesic flow}\label{section geodesic flow}


\subsection{The geodesic flow framework}

Let $(M,g)$ be a surface, i.e. a compact without boundary connected $C^\infty$ Riemannian manifold of dimension $2$, with $g\in\RR^r(M)$. 
Here $\RR^r(M)$ stands for the set of $C^r$ metrics on $M$ with $2 \leq r \leq \infty$. Given a tangent vector $v\in T_xM$ at a point $x\in M$, denote by 
$$
\gamma_{x,v}\colon[0,+\infty)\to M
$$
the geodesic such that $\gamma_{x,v}(0)=x$ and $\dot\gamma_{x,v}(0)=v$.
The {\em geodesic flow} of $g$ is the one-parameter family of diffeomorphisms on the tangent bundle
\begin{equation*}
\begin{split}
\phi_g^t\colon TM & \to TM \\
(x,v) & \mapsto \left(\gamma_{x,v}(t),\dot\gamma_{x,v}(t)\right).
\end{split}
\end{equation*}
Since geodesics travel with constant speed, the unit tangent bundle 
$$
S_gM=\{(x,v)\in TM\colon g_x(v,v)=1 \}
$$ 
is preserved by $\phi_g^t$.
By writing the canonical projection $\pi\colon S_gM\to M$, we have that geodesics $\gamma\subset M$ lift to orbits of the geodesic flow $\pi^{-1}\gamma\subset S_gM$.

It is widely known that the geodesic flow is a Hamiltonian flow given by $(x,v)\mapsto\frac12 g_x(v,v)$ on $TM$ for a symplectic form depending on $g$ (cf.~\cite{Paternain}). 
This is related to another Hamiltonian flow on the cotangent bundle $T^*M$ with a symplectic form which does not depend on the metric and it is defined in the following way.

Let $(x,p)\in T^*M$ and $\eta\in T_{(x,p)}T^*M$. 
Using the canonical projection $\tilde\pi\colon T^*M\to M$, consider the one-form
$$
\lambda(\eta)=p(d\tilde\pi(x,p)\,\eta). 
$$
Notice that in local coordinates this form is simply given by $p\,dx$.
Now, $\omega=-d\lambda$ is a symplectic form on $T^*M$ (in local coordinates, $\omega=dx\wedge dp$).

The Hamiltonian flow can be obtained from the fact that geodesics are solutions of the Euler-Lagrange equation for the Lagrangian $L(x,v)=\frac12 g_x(v,v)$, $(x,v)\in TM$.
Using the Legendre transform $\LL\colon TM\to T^*M$, the problem can be put into the Hamiltonian formalism by writing the Hamiltonian $H=L\circ \LL^{-1}$.


In local coordinates of $M$ we can write for $x\in M$ and $v\in T_xM$ the metric
$$
g_x(v,v)=\langle A(x)^{-1} v,v\rangle,
$$
where $A(x)^{-1}$ is a symmetric positive definite matrix, $x\mapsto A^{-1}(x)$ is $C^\infty$ and $\langle \cdot,\cdot\rangle$ stands for the usual inner product.
The Legendre transformation $\LL\colon TM\to T^*M$ is
$$
\LL(x,v)=(x,A(x)^{-1}v)=(x,p),
$$
i.e. $v=A(x)\,p$.
The Hamiltonian that generates the geodesic flow is then $H\colon T^*M\to\Rr$ with
$$
H(x,p)=\frac12\langle A(x)p,p\rangle.
$$
Notice that $H$ is actually a metric on $T^*M$.

The Hamiltonian vector field generates the orbits of the Hamiltonian flow $\varphi_g^t$ which are the same for every energy level up to a time reparametrization. 
That is, $\varphi_g^{st}(x,p)=\varphi_g^t(x,sp)$.
It is therefore enough to consider one energy level, in particular the invariant unit cotangent bundle
$$
S_g^*M=H^{-1}(1/2).
$$

The flow $\varphi_g^t\colon S_g^*M\to S_g^*M$ is called the {\em Hamiltonian geodesic flow} associated to $g$ since $\tilde\pi\circ \varphi_g^t(x,p)=\gamma_{x,v}(t)$.
Moreover, the relation between the Hamiltonian geodesic flow $\varphi_g^t$ on $T^*M$ and the geodesic flow $\phi_g^t$ on $TM$ is
\begin{equation}\label{Legendre}
\phi_g^t=\LL^{-1}\circ\varphi_g^t\circ \LL.
\end{equation}
We shall denote by $d(\cdot,\cdot)$ the distance function in $S_g^*M$.

\begin{remark}
A local perturbation of the Riemannian metric $g$ supported in a set $U\subset M$, causes a change of the geodesic flow in all fibers in $S_g^*U\subset S_g^*M$ and not just in a neighbourhood in the phase space $S_g^*M$.
This is a major difficulty for the use of local perturbations in the Riemannian metrics setting.
\end{remark}

Since $M$ is compact so is $S_g^*M$.
A \emph{transversal} $\Sigma$ to the flow at a regular point $(x,p)$ in $S_g^*M$ is a $2$-dimensional smooth submanifold verifying
$$
T_{(x,p)}S_g^*M = T_{(x,p)}\Sigma \oplus \Rr X_g(x,p).
$$
Note that $\Sigma$ is a symplectic submanifold.

Consider a $C^1$-family of transversals $\Sigma_t$ to the flow at $\varphi_g^t(x,p)$, $t\geq0$, and of neighborhoods $U_t\subset S_g^*M$ of $(x,p)$.
The \emph{transversal Poincar\'e flow} of $g$ at $(x,p)$ is defined to be the family of $C^{1}$-symplectomorphisms 
$$
P_g^t\colon \Sigma_0\cap U_t\to  
\Sigma_t
$$ 
given by
$
P_g^t(y,q)=\varphi_g^{\Theta(y,q,t)}(y,q)
$
with
$$
\Theta(y,q,t)=\min\{s\geq0\colon \varphi_g^s(y,q) \in\Sigma_t\}.
$$
We assume that $U_t$ is sufficiently small such that, by the implicit function theorem, $\Theta$ is $C^{1}$ and $\Theta(U_t,t)$ is bounded for a fixed $t>0$.

The \textit{transversal linear Poincar\'e flow} of $g$ at $(x,p)$ is the derivative of $P_g^t$ at $(x,p)$,
$$
DP_g^t(x,p)\colon T_{(x,p)} \Sigma_0 \to T_{\varphi_g^t(x,p)}\Sigma_t.
$$ 

Given a regular point $(x,p)$, we say that $(x,p)$ is a \textit{periodic point} of the Hamiltonian geodesic flow $\varphi_g^t$ if $\varphi_g^t(x,p)=(x,p)$ for some $t$. The smallest $t_0>0$ satisfying the condition above is called \emph{period} of $(x,p)$; in this case, we say that the orbit of $(x,p)$ is a \textit{closed orbit} of period $t_0$.
Nontrivial closed geodesics on $M$ for $g$ are in one-to-one correspondence with the closed orbits of $\varphi_g^t$. When $(x,p)$ is periodic of period $\ell>0$ we call $P_g:=P_g^\ell(x,p)$ the \emph{Poincar\'e map} and $\Sigma$ the \emph{Poincar\'e section}.

In~\cite[\S2.3]{MBJLD} it was proved a result for Hamiltonians which can be translated into our context in the following way: a $\varphi_g^t$-invariant regular compact subset $\Lambda\subset S_g^*M$ is uniformly hyperbolic for $\varphi_g^t$ if and only if the associated transversal linear Poincar\'{e} flow $DP^t_g$ is uniformly hyperbolic on the tangent space of $\Sigma_\Lambda$ denoted by $T\Sigma_\Lambda$. With this in mind we define the hyperbolic structures with respect to the transversal linear Poincar\'{e} flow. 

Given a $C^2$-metric  $g$ and a $\varphi_g^t$-invariant, compact and regular  set $\Lambda\subset S_g^*M$, we say that $\Lambda$ is \emph{uniformly hyperbolic} if there exist $\theta\in(0,1)$ and  $m>0$ and a $DP_g^t$-invariant splitting $E_{\Lambda}^s\oplus E_{\Lambda}^u$ of  $T\Sigma_{\Lambda}$ such that for any $(x,p)\in\Lambda$ we have
\begin{center}
$\|DP_g^{m}(x,p)|_{E^s_{(x,p)}}
\|\leq \theta$ and $\|DP_g^{-{m}}(\varphi_g^{m}(x,p))|_{E^u_{\varphi_g^{m}(x,p)}}\|\leq\theta$.
\end{center}

A periodic point $(x,p)$ is called \textit{hyperbolic} if its whole orbit is a uniform hyperbolic set.
 Equivalently, a closed geodesic is hyperbolic if its transversal linear Poincar\'e flow on the period has no eigenvalue of modulus $1$ (notice that the eigenvalues are independent of the choice of the transversal and of the point in the closed orbit). If the eigenvalues are non-real and with modulus $1$ the closed orbit is said to be \emph{elliptic}, and if they are irrational we say that the orbit is \emph{irrationally elliptic}. 
 The \emph{parabolic} closed orbits have real eigenvalues equal to $1$ or $-1$.  It is well-known that for an open and dense subset of metrics on surfaces its geodesic flows display only elliptic or hyperbolic closed orbits (see \cite{C}).

 A {\em locally maximal} invariant set (or {\em isolated set}) is a compact subset 
$\Lambda \subset S_g^*M$ such that $\varphi_g^t(\Lambda)=\Lambda$ for all $t \in \Rr$ and there is a neighbourhood $U$ of $\Lambda$, called {\em isolating block}, such that
$\Lambda=\bigcap_{t \in \Rr} \varphi_g^t(U)$.


\section{Shadowing, weak shadowing and specification}
\label{section: shadowing}

In this section we introduce the dynamical properties that we shall deal with in the sequel.

\subsection{The shadowing property}

Let $\varphi_g^t\colon S_g^*M\to S_g^*M$ be the Hamiltonian geodesic flow associated to the metric $g \in \mathcal{R}^\infty(M)$.
The notion of shadowing developed in ~\cite[\S3.2]{BRT} can be adapted to the geodesic flow. Indeed, in our case it is easier because it is enough to consider a single energy level.

Let us fix real numbers $\delta, T>0$. We say that a pair of
sequences  $[(x_i,p_i), (t_i)]_{i \in \Z}$, where $(x_i,p_i) \in S_g^*M$, $t_i \in \R$, $t_i \geq T$, 
is a
$(\delta,T)$-{\em pseudo-geodesic} of  $\varphi_g^t$ if 
$$d(\varphi_g^{t_i}(x_i,p_i),(x_{i+1},p_{i+1})) < \delta \,\,\text{for all }\,\,i \in \Z.$$
For the sequence $(t_i)_{i \in \Z}$ we write
$\varsigma(n)=t_0+t_1+\ldots+t_{n-1}$ if  $n > 0$, 
$\varsigma(n)=-(t_n+\ldots+t_{-2}+t_{-1})$ if $n < 0$ and $\varsigma(0)=0$. 

Let $(x_0,p_0) \star t$ denote a point on a $(\delta,T)$-chain $t$ units time from $(x_0,p_0)$. More precisely, for $t \in \R$,
$$(x_0,p_0) \star t=\varphi_g^{t-\varsigma(i)}(x_i,p_i) \hspace{0.4cm} {\text{\rm if}} \hspace{0.4cm} \varsigma(i) \leq t <\varsigma(i+1).$$
By $\mbox{Rep}$ we denote the set of all increasing homemorphisms $\tau\colon \R \rightarrow \R$,
called (time) {\em reparameterizations}, satisfying $\tau(0)=0$. Fixing $\varepsilon>0$, we define the set
$$
\mbox{Rep}(\varepsilon)=\left\{\tau \in \mbox{Rep}: \left|\frac{\tau(t) -\tau(s)}{t-s}-1 \right|<\varepsilon, \, s, t \in \R \right\},
$$ 
of the reparameterizations $\varepsilon$-close to the identity.

A $(\delta,T)$-pseudo-geodesic $[(x_i,p_i), (t_i)]_{i \in \Z}$ is $\varepsilon$-{\em shadowed} by some true geodesic of $g$
 if there is $(\tilde x,\tilde p) \in S_g^*M$  and a reparameterization $\tau \in \mbox{Rep}(\varepsilon)$ such that 
\begin{equation}\label{shadow}
d(\varphi_g^{\tau(t)}(\tilde x,\tilde p),(x_0,p_0) \star t)<\varepsilon,\text{ for every }t \in \R.
\end{equation}
The Hamiltonian geodesic flow $\varphi_g^t$ is said to have the {\em shadowing property} if, for any $\varepsilon>0$ there exist $\delta,\, T>0$ such that
any $(\delta,T)$-pseudo-geodesic $[(x_i,p_i), (t_i)]_{i \in \Z}$ is $\varepsilon$-{shadowed} by some geodesic of $g$.
Finally, we say that the Hamiltonian geodesic flow $\varphi_g^t$  is {\em stably shadowable} if
there exists a $C^2$-neighborhood $\mathcal{V}$ of $g$ where for any $C^\infty$-metric $\hat{g} \in \mathcal{V}$ the flow $\varphi_{\hat{g}}^t$  has the shadowing property.

\subsection{The weak shadowing property} \label{wsp}

The shadowing property in the weak sense first appeared in a paper by Corless and Pilyugin (see ~\cite{CPi}) related to the genericity  of shadowing among homeomorphisms, with respect to the $C^0$-topology. In simple terms \emph{weak shadowing} allows to approximate  ``almost orbits"  by true orbits, if  one considers only the distance between the orbit and the ``almost orbit" as two subsets in the manifold, thus forgetting  the time parameterization. There exist dynamical systems without the weak shadowing property (see ~\cite[Example 2.12]{P}) and dynamical systems satisfying the weak shadowing property but not the shadowing one (\cite[Example 2.13]{P}).

We recall the following definition of weakly shadowable systems and observe that the first result related to ours was done by Sakai (see \cite{Sa} and the references therein). Given a Hamiltonian geodesic flow $\varphi_g^t\colon S_g^*M\to S_g^*M$ associated to the metric $g \in \mathcal{R}^\infty(M)$ and $\delta, T>0$, a
$(\delta,T)$-pseudo-geodesic $[(x_i,p_i), (t_i)]_{i \in \Z}$ is {\em weakly} $\varepsilon$-{\em shadowed} by some true geodesic of $g$
if there exists $(\tilde x, \tilde p) \in S_g^*M$ such that $\{(x_i,p_i)\}_{i \in \Z} \subset B_\varepsilon(\mathcal{O}(\tilde x, \tilde p))$, where $\mathcal{O}(\tilde x, \tilde p)$ stands for the orbit of $(\tilde x, \tilde p)$.

The Hamiltonian geodesic flow $\varphi_g^t$ is said to have the {\em weak shadowing property} if, for any $\varepsilon>0$ there exist $\delta, T>0$ such that
any $(\delta,T)$-pseudo-geodesic $[(x_i,p_i), (t_i)]_{i \in \Z}$ is weakly $\varepsilon$-shadowed by some geodesic of $g$.

Finally, we say that the Hamiltonian geodesic flow $\varphi_g^t$  is {\em stably weakly shadowable} if
there exists a $C^2$-neighbourhood $\mathcal{V}$ of $g$ where for any $C^\infty$-metric  $\hat{g} \in \mathcal{V}$ the flow $\varphi_{\hat{g}}^t$ has the weak shadowing property.

\subsection{The specification property}\label{spec}

Consider a Hamiltonian geodesic flow $\varphi_g^t\colon S_g^*M\to S_g^*M$ associated to the metric $g \in \mathcal{R}^\infty(M)$
 and a 
$\varphi_g^t$-invariant compact set $\Lambda \subset S_g^*M$.

A {\em specification} $\mathcal{S}=(\sigma,P)$ consists of a finite collection
$\sigma=\{I_1,\ldots,I_m\}$ of bounded disjoint intervals $I_i=[a_i,b_i]$
of the real line and a map $P\colon\bigcup_{I_i \in \sigma}I_i \rightarrow \Lambda$ such that for any $t_1,t_2 \in I_i$ we have
$$\varphi^{t_2}_g(P(t_1))=\varphi^{t_1}_g(P(t_2)).$$
The specification $\mathcal{S}$ is said to be $K$-{\em spaced} if $a_{i+1} \geq b_i+K$ for all $i \in \{1,\cdots,m\}$ and the minimal of such
$K$ is called the spacing of $\mathcal{S}$. If $\sigma=\{I_1,I_2\}$, then $\mathcal{S}$ is said to be a {\em weak specification}.
Given $\varepsilon>0$, we say that $\mathcal{S}$ is $\varepsilon$-{\em shadowed} by $(x,p) \in \Lambda$ if $d(\varphi_g^t(x,p),P(t))<\varepsilon$
for all $t \in \bigcup_{I_i \in \sigma}I_i$.

We say that $\Lambda$ has the {\em weak specification property} if for any $\varepsilon>0$
there exists a $K=K(\varepsilon) \in \Rr$ such that any $K$-spaced weak specification $\mathcal{S}$ is
$\varepsilon$-shadowed by a point of $\Lambda$. 
In this case the Hamiltonian geodesic flow $\varphi_g^t |_{\Lambda}$ is said to have the weak specification property.
The Hamiltonian geodesic flow $\varphi_g^t$ is said to have the weak specification property if $S_g^*M$ has it.

We say that the Hamiltonian geodesic flow $\varphi_g^t$ associated to $g$ has the {\em stable weak specification property} if
there exists a $C^2$-neighbourhood $\mathcal{V}$ of  $g$ where for any $C^\infty$-metric $\hat{g} \in \mathcal{V}$ the flow $\varphi_{\hat{g}}^t$  has the weak specification property.

\section{Statement of the results}\label{State}


Given $g \in \mathcal{R}^{\infty}(M)$, let
$$\mathscr{P}(g):=\{\gamma\colon \gamma \,\text{is a closed orbit under } \varphi_g^t\} $$
and
$$
Per(g):=\underset{\gamma\in\mathscr{P}(g), \, t\in\mathbb{R}}{\bigcup}{\gamma(t)}.$$

\begin{maintheorem}\label{thm1}
Let $M$ be a surface. If $g \in \mathcal{R}^\infty(M)$
and the Hamiltonian geodesic flow $\varphi_g^t$ satisfies one of the properties:
\begin{itemize}
\item [(a)] is stably shadowable; 
\item [(b)] is stably weak shadowable; 
\item [(c)] has the stable weak specification property;
\end{itemize}
then $\overline{Per(g)}$ is a uniformly hyperbolic set.
\end{maintheorem}


The proof of Theorem~\ref{thm1} is an immediate consequence of Theorems~\ref{thm2} and~\ref{thm3} below.

Consider the \emph{Ma\~n\'e star systems} defined by the $C^2$-interior of the metrics such that all closed orbits are hyperbolic:
$$\mathscr{H}(M):=\{g\in \mathcal{R}^{\infty}(M)\colon \text{any}\,\gamma\in\mathscr{P}(g)\,\text{is hyperbolic} \} 
$$
and
$$ 
\mathscr{F}^2(M):=\mbox{int}_{C^2} \mathscr{H}(M).$$ 
Clearly, $g \in \mathscr{F}^2(M)$ means that  $g \in \mathscr{H}(M)$  and for any $\hat g \in \mathcal{R}^{\infty}(M)$, $C^2$-arbitrarily close to $g$, we also have that  $\hat g \in \mathscr{H}(M)$. 

\begin{maintheorem}\label{thm2}
Let $M$ be a surface. If $g\in \mathcal{R}^\infty(M)$ and the Hamiltonian geodesic flow $\varphi_g^t$ satisfies one of the properties:
\begin{itemize}
\item [(a)] is stably shadowable; 
\item [(b)] is stably weak shadowable; 
\item [(c)] has the stable weak specification property;
\end{itemize} 
then $g \in \mathscr{F}^2(M)$.
\end{maintheorem}

The proof is contained in section~\ref{proofs}.
We point out that the proof of the Hamiltonian version of this theorem in ~\cite{BRT} uses a suspension theorem ~\cite{BD2} which is unavailable for geodesic flows.


\begin{maintheorem}\label{thm3} (\cite[Theorem D]{CP})
Let $M$ be a surface. If $g \in \mathscr{F}^2(M)$, then $\overline{Per(g)}$ is a uniformly hyperbolic set.
\end{maintheorem}

Notice also that the general Hamiltonian version of Theorem~\ref{thm3} contained in \cite{BRT} is stronger because it requires the closing lemma, unknown for geodesic flows.


\section{Proof of Theorem~\ref{thm2}}\label{proofs}

Theorem ~\ref{thm2} is an immediate consequence of the following results.

\begin{proposition}\label{elliptic}
Let $M$ be a surface and $g \in \mathcal{R}^\infty(M)$. If the Hamiltonian geodesic flow $\varphi_g^t$ satisfies one of the properties:
\begin{itemize}
\item [(a)] is shadowable; 
\item [(b)] is weak shadowable; 
\item [(c)] has the weak specification property;
\end{itemize}
then there are no irrationally elliptic closed orbits.
\end{proposition}

Let $M$ be a surface.
Given a simple closed curve $\gamma\in T^*M$, we define the set of $C^\infty$- metrics that have $\gamma$ as an orbit of $\varphi_g^t$ by
$$
\mathcal{R}_\gamma^\infty(M)=\{g\in\mathcal{R}^\infty(M)\colon \gamma\in \mathscr{P}(g)\}.
$$
Endow this set with the $C^2$-topology and let
$$
B_{\varepsilon,\gamma}(g,D) = \{g'\in \mathcal{R}_\gamma^\infty(M)\colon
\|g'-g\|_{C^2}<\varepsilon, \,g=g'\text{ on } D\}.
$$

Moreover, for any $g\in\mathcal{R}_\gamma^\infty(M)$ consider the map 
$$
T_\gamma\colon g\mapsto \tr {DP_g}|_\gamma
$$ 
that gives the trace of the transversal linear Poincar\'e flow at $\gamma$.
Below we use also the notation $B_\delta(a)=\{y\in\Rr\colon |y-a|<\delta\}$.

\begin{lemma}\label{local perturbation}
Let $\varepsilon>0$, $g\in\mathcal{R}^\infty(M)$ and $\gamma$ be a closed orbit for $\varphi_g^t$. 
Then, there is $\delta>0$ such that for any tubular neighbourhood $W\subset M$ of $\widetilde\pi\gamma$,
$$
B_\delta(T_\gamma (g))\subset T_\gamma(B_{\varepsilon,\gamma}(g,D)),
$$
where $D=(M\setminus W)\cup \widetilde\pi\gamma$.

\end{lemma}

\begin{proof}
This follows from the version of the Franks' lemma for geodesic flows on surfaces~\cite{CP,V}.
\end{proof}

If we have a parabolic or elliptic closed orbit of the Hamiltonian geodesic flow for a given metric, by Lemma~\ref{local perturbation} there is a nearby metric with the same closed orbit but irrationally elliptic. Proposition~\ref{elliptic} then implies that the shadowing properties can not stably hold. 
This proves Theorem~\ref{thm2}.

It remains to show Proposition~\ref{elliptic}.


\subsection{Proof of Proposition~\ref{elliptic} \mbox{(a)} and \mbox{(b)}}
\label{section proof prop}

The existence of an irrationally elliptic closed orbit implies the existence of invariant curves around the corresponding fixed point of the Poincar\'e map. These curves split the Poincar\'e section and are a clear obstacle to shadowing.
We present below the details of the proof.

Assume that $\varphi_g^t$ has an irrationally elliptic closed point $(x,p)$ which corresponds to a fixed point of the Poincar\'e map $P_g$ defined on a transversal section $\Sigma$ at $(x,p)$.

Since the eigenvalues of $DP_g(x,p)$ are irrational (non-resonant), the Birkhoff normal form theorem gives us a good coordinate transformation on a small neighbourhood $U$ of $(x,p)$ that reduce $P_g$ in $U\cap \Sigma$ to an area-preserving twist map (see e.g. \cite[Proposition 38.4]{Gole} or \cite{Mo}).
Moreover, by introducing symplectic polar coordinates, we obtain an area-preserving twist map $Q_g=h\circ P_g\circ h^{-1}$ on $\Aa=\Tt \times [0,+\infty)$, $\Tt=\Rr/\Zz$, where $h\colon \Sigma\to\Aa$ is the full coordinate transformation.
This map yields all the dynamics in the neighbourhood of the elliptic fixed point ($r=0$) and it is given by
\begin{equation}\label{twist}
Q_g(\theta,r)= (\theta + \tau r+ F(\theta,r) \bmod 1, 
r+ G(\theta,r)),
\end{equation}
where $\tau\not=0$, and $F,G$ are small $C^\infty$ functions of order $r$. 
Notice that $r=0$ is a segment of fixed points.

It is well-known that any invariant Jordan curve $\Gamma$ homotopically nontrivial is the graph of a Lipschitz function $\psi\colon\Tt\to\Rr$, i.e. 
$$
f(\theta,\psi(\theta))=(\phi(\theta),\psi\circ\phi(\theta)),
\quad
\theta\in\Tt,
$$ 
where $\phi$ is a homeomorphism of $\Tt$ with rotation number $\rho(\Gamma)$.
We call such sets {\em invariant circles}.

Denote the set of all invariant circles by $K$. This set is not empty, in fact it contains a positive measure set consisting of smooth curves given by KAM theory (c.f. e.g.~\cite{Moser}) with all points in $r=0$ being density points.
This implies that in any neighborhood of $r=0$ there is a smooth invariant circle with a diophantine rotation number (KAM circle). 
In addition, every KAM circle is accumulated from above and below by other KAM circles.

If a connected component of the complement of $K$ is homeomorphic to an annulus, i.e. the boundary is the union of two disjoint invariant curves $\Gamma_-,\Gamma_+$ (which can not be KAM circles), it is called a {\em Birkhoff zone of instability}.
Otherwise, the connected component of the complement of $K$ corresponds to a chain of heteroclinic orbits to hyperbolic periodic points. Thus, the boundaries are invariant curves with the same rational rotation number, intersecting at the hyperbolic periodic points (see e.g.~\cite{Herman,Calvez}).
This last case does not hold for generic area-preserving twist maps.

The result below guarantees the existence of orbits whose closure connects the boundaries of a given Birkhoff zone of instability.

\begin{proposition}(Herman~\cite[\S 5.9.2, \S 5.9.3, \S 5.9.4]{Herman})\label{KAM2}
Consider a Birkhoff zone of instability bounded by the invariant circles $\Gamma_-$ and $\Gamma_+$ of any rotation number and let $(\theta,r)\in\Gamma_-$.
Then, for every neighbourhood $W$ of $(\theta,r)$ we have that
$$
\Gamma_+ \cap \overline{\bigcup_{n\in\Zz}Q_g^n(\theta,r)(W)} \not=\emptyset.
$$
\end{proposition}

We will show that there are pseudo-orbits in a neighbourhood of $r=0$ which are not possible to shadow.
Take $\Gamma_0=\{r=0\}$, two KAM circles $\Gamma_1$ and $\Gamma_2$ near $\Gamma_0$
and
$$
\varepsilon'=\frac12
\min_{i\not=j}d(\Gamma_i,\Gamma_j).
$$
The above curves are graphs of the Lipschitz functions $\psi_i$ with $\psi_0=0<\psi_1<\psi_2$.

Any orbit can not be $\varepsilon'$-close to more than two of the above invariant sets.
We are going to construct a pseudo-orbit which reaches the three curves.

Start at any point $(\theta_0,r_0)\in\Gamma_0$, i.e. $r_0=0$.
Let the backward pseudo-orbit be the real backward orbit $(\theta_n,r_n)=(\theta_0,0)\in\Gamma_0$ for $n<0$.
The forward pseudo-orbit will be built in order to end up at or above $\Gamma_2$ and that the ``jumps'' (distance at the same iterate between the real orbit and the pseudo-orbit) occur only at isolated times spaced by any given interval.

Consider the canonical projection $\pi(\theta,r)=r$.
We choose $(\theta_1,r_1)$ to be $\delta'$-close to $Q_g(\theta_0,r_0)$ and 
$$
0 < r_1-\pi\circ Q_g(\theta_0,r_0) < \delta',
$$
so that $(\theta_1,r_1)$ is either in $\Gamma_1$ (in case $\delta'$ is large enough) or else it is in an invariant curve $\Gamma$ strictly between $\Gamma_0$ and $\Gamma_1$.

If $\Gamma$ is not a lower boundary of an instability zone or part of a chain of hyperbolic heteroclinic orbits, the pseudo-orbit can stay in $\Gamma$ as long as required and at any iterate increase the $r$-component up to $\delta'$ repeating the above procedure.

If $\Gamma$ is an invariant circle part of a chain of hyperbolic heteroclinic orbits, choose the pseudo-orbit as the real orbit that follows the dynamics throught the stable manifold as long as required. Close to the hyperbolic periodic point jump to the region above the chain.

Finally, if $\Gamma$ is the lower boundary of a Birkhoff zone of instability, we use the dynamics and Proposition~\ref{KAM2} so that the pseudo-orbit reaches the top boundary. 
The pseudo-orbit is equal to the real orbit on $\Gamma$ for any required time.
Then, take a small neighbourhood $W$ of a point in $\Gamma$ and jump into a point there whose forward iterate gets close enough to the upper boundary of the zone of instability.

In this way we are able to reach $\Gamma_1$ (and also $\Gamma_2$) at some iterate. From that time on consider the real orbit on $\Gamma_2$. We have thus constructed a pseudo-orbit $\{(\theta_n,r_n)\}_{n\in\Zz}$ that is equal to the real orbit for any given finite segment of the orbit, connecting $\Gamma_0$ to $\Gamma_2$.

Take now $(x_n,p_n)=h^{-1}(\theta_n,r_n)\in\Sigma$, $n\in\Zz$, and $t_n=\Theta(x_n,p_n,\ell)$, where $\Theta$ is the first return time to $\Sigma$ and $\ell$ is the period of the periodic orbit. Notice that $t_n$ is close to $\ell$ and bounded away from zero.
This defines a $(\delta,T)$-pseudo-geodesic of $g$ for given $\delta,T>0$.
Notice that $\delta'$ above is related to $\delta$ and $T$ gives us a lower bound on the number of consecutives iterates without jumps.

In conclusion, we have shown that there is $\varepsilon>0$ such that for any $\delta,T>0$ we can find a $(\delta,T)$-pseudo-geodesic which is not $\varepsilon$-shadowed by any true geodesic of $g$.

The above construction also implies that the pseudo-geodesic does not have the weak shadowing property.


\subsection{Proof of Proposition~\ref{elliptic} \mbox{(c)}}

We say that $\varphi_g^t$ is {\em topologically mixing} if for any given open sets $U,V\subset S_g^*M$ we can find $N >0$ such that 
$$
U \cap \varphi_g^t(V)  \neq \emptyset,\quad t \geq N.
$$
Notice that topologically mixing implies transitivity.
The next lemma is a particular case of~\cite[Lemma 3.1]{ASS}.

\begin{lemma}\label{mixing}
Let $g \in \mathcal{R}^\infty(M)$.
If $\varphi_g^t$ has the weak specification property, then
it is topologically mixing.
\end{lemma}

For surfaces we are able to show that topologically mixing excludes the existence of elliptic closed orbits, thus completing the proof of Proposition~\ref{elliptic} \mbox{(c)}.

\begin{lemma}\label{KTRobTrans}
Let $M$ be a surface.
If $g \in \mathcal{R}^\infty(M)$ and $\varphi_g^t$ is topologically mixing, then there are no irrationally elliptic closed orbits.
\end{lemma}

\begin{proof}
Assume that $\varphi_g^t$ has an irrationally elliptic closed orbit and the corresponding Poincar\'e (twist) map has an irrationally elliptic fixed point.
Again as in \S\ref{section proof prop}, the existence of invariant curves surrounding the fixed point contradicts transitivity.
Thus, $\varphi_g^t$ is not topologically mixing.
\end{proof}


\section{Applications}\label{section applications} 

\subsection{A generic approach}

\subsubsection{Obtaining a nontrivial hyperbolic set} By using \cite[Theorem 1.1]{CP} we obtain, for a dense subset of metrics, that $\overline{Per(g)}$ is a \emph{nontrivial} uniformly hyperbolic set. Thus, under a dense assumption, we can go further on the conclusions of Theorem~\ref{thm1} and Theorem~\ref{thm3}. 
Still, we observe that the conclusions of Theorem~\ref{thm1} and Theorem~\ref{thm3} can be upgraded even if we consider a generic setup. This was treated with detail in \cite[Theorem D]{C}. In fact, considering a metric in $\mathscr{H}$, where $\mathscr{H}$ is the \emph{residual} subset of strongly bumpy metrics and satisfying a transversality property, if $g\in \mathscr{H}\cap \mathscr{F}^2(M)$, then  $\overline{Per(g)}$ is a \emph{nontrivial} uniformly hyperbolic set. This slightly improves Theorem~\ref{thm1}.

\subsubsection{Shadowing and pointwise hyperbolicity}

In previous section we obtained a generic result under the stability of shadowing (or the weak shadowing or even the specification property), now we will obtain another generic result without this stability assumption.

For that let us consider the residual subset of metrics $\mathscr{R}$ such that (i) all closed orbits are hyperbolic or elliptic and (ii) all elliptic closed orbits are irrationally elliptic. By Proposition~\ref{elliptic} we obtain that under the shadowing hypothesis on $g\in\mathscr{R}$ only hyperbolic closed orbits are allowed. Therefore, we obtain the following result:

\begin{maincorollary}\label{generic}
There exists a $C^\infty$-residual subset $\mathscr{R}\subset\mathcal{R}^\infty(M)$ such that for any $g\in\mathscr{R}$, if $g$ satisfies the shadowing property (or the weak shadowing or even the specification property), then all closed orbits are hyperbolic.
\end{maincorollary}

It is an interesting question to know if the closure of the hyperbolic closed orbits on Corollary ~\ref{generic} is a uniformly hyperbolic set, say switch pointwise hyperbolicity by uniform hyperbolicity on the set of closed orbits.

\subsection{Hyperbolic homoclinic classes}

As discussed at the end of section \ref{introduction} it is not known if the closing lemma with respect to the $C^2$-topology on the metric holds for geodesic flows. Hence, on any manifold we are not sure if we have dense closed orbits for $C^2$-dense metrics. Clearly, if the manifold has negative curvature, the flow is Anosov and so, it displays dense closed orbits without needing any perturbation (cf. ~\cite{A}). Next we consider certain invariant \emph{a priori} proper subsets on the surface $M$ with dense closed orbits and show that they spread to the whole manifold under $C^2$-stability of shadowing. We will prove that if these sets have the shadowing (or the weak shadowing or even the specification property) property $C^2$-robustly, then these sets are actually the whole manifold and so the closed orbits are abundant in $M$. In overall, we obtain a sufficient condition to obtain closing without the need of any perturbation.

We recall some basic definitions. We define the strong stable and stable manifolds of $(x,p)$ as:
$$\displaystyle{W^{ss}(x,p):=\{(\tilde x, \tilde p) \in S_g^*M: \lim_{t \rightarrow +\infty} d(\varphi_g^t(\tilde x, \tilde p),\varphi_g^t(x,p))=0\}}$$
which, when $(x,p)$ is hyperbolic, is a $1$-dimensional set
and
$$\displaystyle{W^{s}(\mathcal{O}(x,p)):=\bigcup_{t \in \Rr} W^{ss}(\varphi_g^t(x,p))},$$
which, when $(x,p)$ is hyperbolic, is a $2$-dimensional set
  where $\mathcal{O}(x,p)$ stands for the orbit of $(x,p)$.
For small $\varepsilon>0$, the local strong stable manifold is  an embedded disk contained in the global stable manifold $W^{ss}(x,p)$ and is defined as
$$\displaystyle{W^{ss}_{\varepsilon}(x,p):=\{(\tilde x, \tilde p) \in S_g^*M: \, d(\varphi_g^t(\tilde x, \tilde p),\varphi_g^t(x,p))<\varepsilon \,\, \text{if} \,\, t \geq 0\}}.$$
By the stable manifold theorem, there exists an $\varepsilon=\varepsilon(x,p)>0$ such that
$$\displaystyle{W^{ss}(x,p)=\bigcup_{t \geq 0} \varphi_g^{-t}(W^{ss}_{\varepsilon}(\varphi_g^t(x,p)))}.$$
Analogous definitions hold for unstable manifolds.

 Any point $(\hat x,\hat p)\not=(x,p)$
 in $W^{uu}(x,p)\cap W^{ss}(x,p)$ is called a \emph{homoclinic point}. This intersection is said to be transversal if the dimension of the subspace spanned by $T_{(\hat x,\hat p)}W^{su}(x,p)$ and $T_{(\hat x,\hat p)}W^{ss}(x,p)$ is equal to $2$.

Given a hyperbolic point $(x,p)\in S_g^*M$ for $\varphi_g^t$ its \emph{homoclinic class}, denoted $H((x,p),\varphi_g^t)$, is the closure of the set of transverse intersections between the stable and unstable manifolds of all points in the orbit of $\varphi_g^t(x,p)$. It is well-known by the Birkhoff-Smale Theorem that $H((x,p),\varphi_g^t)$ is a transitive set with dense closed orbits.

All the definitions of shadowing, weak shadowing and specification given previously can be readapted to a local point of view by considering those properties defined in isolated sets. Next we define them properly so we have no ambiguity.
\begin{itemize}
\item (shadowing) The Hamiltonian geodesic flow $\varphi_g^t$ is said to have the {\em shadowing property in the isolated set $\Lambda$} if, for any $\varepsilon>0$ there exist $\delta,\, T>0$ such that
any $(\delta,T)$-pseudo-geodesic $[(x_i,p_i), (t_i)]_{i \in \Z}$ in $\Lambda$ is $\varepsilon$-{shadowed} by some geodesic of $g$. We say that the Hamiltonian geodesic flow $\varphi_g^t$ associated to $g$ is {\em stably shadowable in $\Lambda$} if there exists an isolating block $U$ such that the flow $\varphi_{\hat g}^t$ displays the shadowing property in $\Lambda_{\hat g}:=\bigcap_{t\in\mathbb{R}}\varphi^t_{\hat g}(U)$ for any $C^\infty$ metric $\hat g$ sufficiently $C^2$-close to $g$.

\item (specification property) The Hamiltonian geodesic flow $\varphi_g^t$ is said to have the {\em weak specification property in the isolated set $\Lambda$} if $\Lambda$ has it.
We say that the Hamiltonian geodesic flow $\varphi_g^t$ associated to $g$ has the {\em stable weak specification property in $\Lambda$} if
there exists an isolating block $U$ such that the flow $\varphi_{\hat g}^t$  has the weak specification property in $\Lambda_{\hat g}:=\bigcap_{t\in\mathbb{R}}\varphi^t_{\hat g}(U)$, for any $C^\infty$ metric $\hat g$ sufficiently $C^2$-close to $g$.

\item (weak shadowing property) The Hamiltonian geodesic flow $\varphi_g^t$ is said to have the {\em weak shadowing property in the isolated set $\Lambda$} if, for any $\varepsilon>0$ there exist $\delta, T>0$ such that
any $(\delta,T)$-pseudo-geodesic $[(x_i,p_i), (t_i)]_{i \in \Z}$ in $\Lambda$ is weakly $\varepsilon$-shadowed by some geodesic of $g$.
We say that the Hamiltonian geodesic flow $\varphi_g^t$  is {\em stably weakly shadowable in $\Lambda$} if
there exists an isolating block $U$ such that 
the flow $\varphi_{\hat g}^t$ displays the weak shadowing property in $\Lambda_{\hat g}:=\bigcap_{t\in\mathbb{R}}\varphi^t_{\hat g}(U)$, for any $C^\infty$ metric $\hat g$ sufficiently $C^2$-close to $g$.

\end{itemize}

In this section we obtain the following result:

\begin{maintheorem}\label{New}
If $\varphi_g^t$ is the Hamiltonian geodesic flow on $S_g^*M$, where $M$ is a surface, and $\Lambda\subset S_g^*M$ is a homoclinic class satisfying the shadowing  property (or weak shadowing, or specification) $C^2$-robustly, then $\varphi_g^t$ is Anosov.
\end{maintheorem}


The results that we proved before can be easily adapted in order to obtain the following result. We leave the details of the proof to the reader.

\begin{lemma}
If $\varphi_g^t$ is the Hamiltonian geodesic flow on $S_g^*M$, where $M$ is a surface, and $\Lambda\subset S_g^*M$ is a homoclinic class satisfying the shadowing  property (or weak shadowing, or specification) $C^2$-robustly, then $\Lambda$ is hyperbolic.
\end{lemma}

Then, the proof of Theorem~\ref{New} is derived directly from the following lemma which is based in a simple but beautiful idea of Newhouse (\cite{New}).

 \begin{lemma}
If $\varphi_g^t$ is the Hamiltonian geodesic flow on $S_g^*M$, where $M$ is a surface, and $\Lambda\subset S_g^*M$ is a hyperbolic homoclinic class, then $\Lambda=S_g^*M$ and $\varphi_g^t$ is Anosov.
\end{lemma}

\begin{proof}
Since $\Lambda$ is clearly closed and $S_g^*M$ is connected it is sufficient to prove that $\Lambda$ is open. For any $(x,p)\in\Lambda$ we will show that there exists a product structure around $(x,p)$ formed by stable/unstable local manifolds of uniform size. Knowing that the geodesic flow has constant velocity (in particular do not have equilibrium points) is crucial to go on with the proof and avoid singular-hyperbolicity thus non-uniform sizes of invariant manifolds. Given $W^{uu}(x,p)$ we claim that densely in $W^{uu}(x,p)$ we have elements in $\Lambda$ and analogously in $W^{ss}(x,p)$ which is sufficient to obtain the product structure. Suppose, by contradiction, that we have a hole in $W^{uu}(x,p)$ without elements in $\Lambda$, say an open interval $(a,b)\subset W^{uu}(x,p)$ without elements in $\Lambda$. Since the periodic points are dense in $\Lambda$ we can choose $(\hat x,\hat p)\in Per(g)\cap \Lambda$ very close to $(x,p)$. By continuity of the unstable manifold locally we have that $W^{uu}(\hat x,\hat p)$ and $W^{uu}(x,p)$ are $C^1$-close. Now, by invariance of $W^{uu}(x,p)$ we can transport the hole $(a,b)$ near $(x,p)$ in order to obtain a more or less straight rectangular box $\mathcal{B}$, see Figure ~\ref{Fig1}, formed by:
\begin{figure}\begin{center}
\includegraphics[height=5cm]{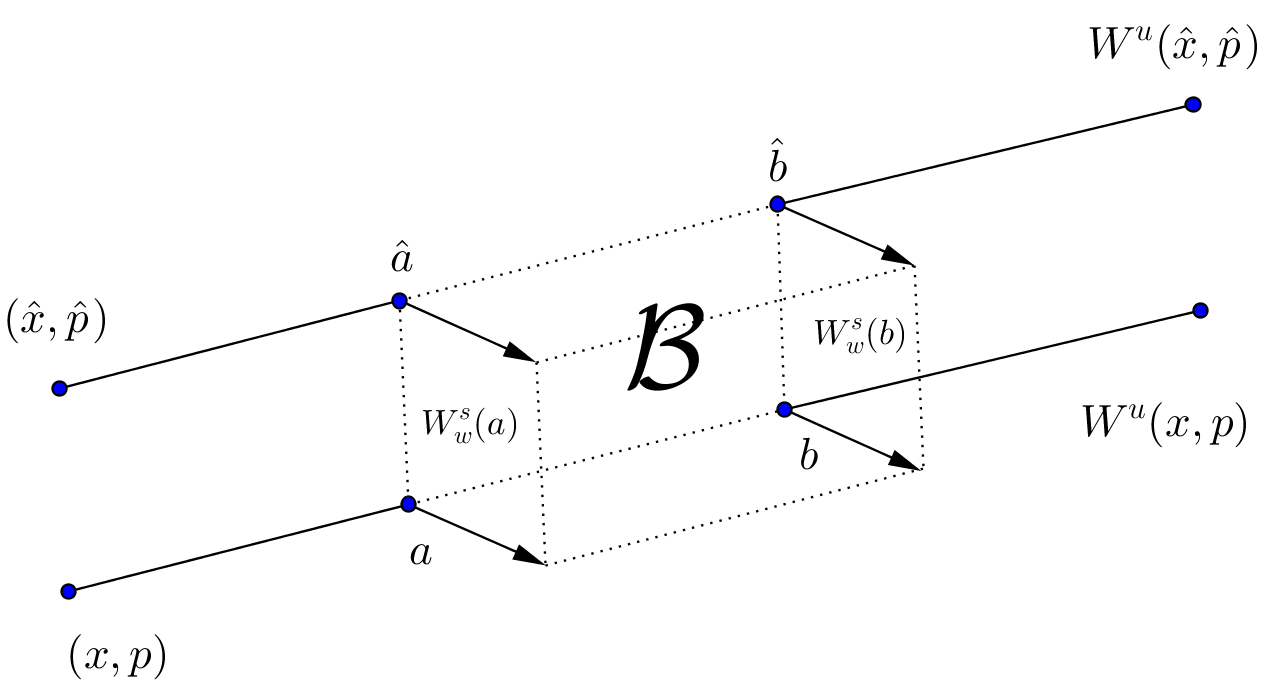}
\end{center}
\caption{Construction of the box $\mathcal{B}$.}
\label{Fig1}
\end{figure}
\begin{itemize}
\item the base is a $\varphi_g^t$-iteration of $(a,b)$ with time between $[0,\tau]$ for $\tau>0$ very small;
\item the sides are formed by the weak local stable manifolds of $a$ and $b$;
\item the top is a $\varphi_g^t$-iteration of $(a',b')$ with time between $[0,\tau]$ for $\tau>0$ very small where $a'=W_\delta^{ss}(a)\cap W^{uu}(\hat x,\hat p)$ and $b'=W_\delta^{ss}(b)\cap W^{uu}(\hat x,\hat p)$.
\end{itemize}
By Poincar\'e recurrence theorem we have that Lebesgue almost every point in $\mathcal{B}$ return to $\mathcal{B}$. Let $T>0$ be a return point for $y\in \mathcal{B}$. Now, $\varphi_g^t$ will enlarge the length, decrease the height and keep the depth of $\mathcal{B}$.  Since invariant manifolds do not have self intersections and the only way that $\varphi_g^t(\mathcal{B})$ intersects $\mathcal{B}$ is by having an intersection of a $\varphi_g^t$-iterate of an element in the base of $\mathcal{B}$ with the sides of $\mathcal{B}$. But, this would imply that there are elements in $\Lambda$ in the base of $\mathcal{B}$ which is a contradiction with the existence of a hole.

\end{proof}



\section*{Acknowledgements}

MB was partially supported by FCT - `Funda\c{c}\~ao para a Ci\^encia e a Tecnologia', through CMA-UBI, project UID/MAT/00212/2013. 

JLD was partially supported by the project CEMAPRE - UID/MULTI/00491/2013 financed by FCT/MEC through national funds. 

MJT was partially supported by the Research Centre of Mathematics of the University of Minho with the Portuguese Funds from the `Funda\c c\~ao para a Ci\^encia e a Tecnologia', through the Project UID/MAT/00013/2013.


\end{document}